\documentclass{article}
\usepackage{graphicx}
\usepackage{epstopdf}
\usepackage{amsmath}
\usepackage{amssymb}
\usepackage{bm}
\usepackage{bbm}
\usepackage{enumerate}
\usepackage[T1]{fontenc}
\usepackage[latin1]{inputenc}
\usepackage{hyperref}
\usepackage{color}
\usepackage{curves}
\usepackage{tikz}
\usepackage[hang]{caption}
\usepackage{ntheorem}
\usepackage{float}
\usepackage{mathrsfs}

\title{Segregating Markov chains}
\author{Timo Hirscher
	\thanks{Research supported by grants from the Swedish Research Council and the Royal Swedish Academy of Sciences}
\qquad Anders Martinsson
     \thanks{Research supported by grants from the Swedish Research Council}
\\\normalsize Chalmers University of Technology and University of Gothenburg}

\theoremstyle{break}
\newtheorem{theorem}{Theorem}[section]

\newtheorem{lemma}{Lemma}[section]
\newtheorem{proposition}{Proposition}[section]
\theorembodyfont{\upshape}
\newtheorem{definition}{Definition}
\newtheorem{remark}{Remark}
\newtheorem{example}{Example}[section]

\makeatletter
\let\c@proposition\c@theorem
\let\c@lemma\c@theorem
\let\c@corollary\c@theorem
\makeatother

\newenvironment{proof}{\noindent{\sc Proof:}}{\vspace{-0.5cm}~\hfill $\square$\vspace{0.5cm}}

\newcommand\N{\mathbb{N}}

\newcommand\Prob{\mathbb{P}}
\renewcommand\epsilon{\varepsilon}
\renewcommand\phi{\varphi}

\newcommand{\abs}[1]{\ensuremath{\left| #1 \right|}}

\definecolor{darkblue}{rgb}{0,0,.5}
\hypersetup{colorlinks=true, breaklinks=true, linkcolor=blue, 
citecolor=darkblue, menucolor=blue, urlcolor=blue}

\let\oldvec\vec

\begin{document}
\newpage
\maketitle
%%%%%%%%%%%%%%%%%%%%%%%%%%%
% abstract, keywords and Subject classification are optional.
%%%%%%%%%%%%%%%%%%%%%%%%%%%
\begin{abstract}
Dealing with finite Markov chains in discrete time, the focus often lies
on convergence behavior and one tries to make different copies of the chain
meet as fast as possible and then stick together. There are, however, discrete finite
(reducible) Markov chains, for which two copies started in different states can be coupled
to meet almost surely in finite time, yet their distributions keep a total variation
distance bounded away from 0, even in the limit as time tends to infinity. We show that
the supremum of total variation distance kept in this context is $\tfrac12$.
\end{abstract}

% Most people don't use these, so they are "commented out"
% by starting the lines with a "%"
%\begin{keywords}
%   \LaTeX, typesetting
%\end{keywords}

%\begin{AMS}
%   60J10, 60C05
%\end{AMS}

%%%%%%%%%%%%%%%%%%%%%%
% % Here is the start of the Text
%%%%%%%%%%%%%%%%%%%%%%
\section{Introduction}

When the long time behavior of Markov chains is analyzed, one of the most common strategies is to couple
several copies of the chain started in different states. In doing so, one standard approach is to define
two copies of a Markov chain (started in different states) on a common probability space, correlated in
such a way that they are likely to meet within some moderate time, and glue them together
as soon as this happens.

This idea is so predominant that little attention was directed away from such couplings; in the
standard reference \cite{Mixing} it was even claimed erroneously that {\em any} coupling of two Markov
chains with the same transition probabilities can be modified so that the two chains stay together at
all times after their first simultaneous visit to a single state. A counterexample to this statement
was in fact given in \cite{faithful}: If a coupling of two copies of the same Markov chain is changed
in such a way that the second copy mimicks the behavior of the first one once they meet, the altered 
individual process might no longer be a copy of the given chain.

For the sake of simplicity, we want to restrict our considerations to time-homogeneous Markov chains
evolving in discrete time and on countable state spaces -- except for Remark \ref{general} and Theorem
\ref{generalthm}, where we discuss how the argument used to derive Theorem \ref{upperboundthm} applies to more general settings as well.
So let ${\mathbf X}=(X_n)_{n\in\N_0}$ denote a Markov chain on a countable state space $S$ with
transition probabilities $\{P(r,s)=\Prob(X_{n+1}=s\,|\,X_n=r);\ r,s\in S,\ n\in\N_0\}$. While $\mathcal{L}(X_n)$ will be used as shorthand notation for the distribution of $X_n$ in general, we will denote the distribution of $X_n$ given $X_0=x$, i.e.\ for a copy of the chain started in $x\in S$, by $P^n(x,.)$.

In what follows, we want to describe and investigate the kind of Markov chain, that was first introduced
and analyzed by H\"aggstr\"om \cite{disagreement}: A chain in which two copies started in different states
can be coupled such that they almost surely meet, but their distributions do not come arbitrarily close to
one another with respect to total variation distance (cf.\ Definition \ref{TVdef}). This phenomenon -- that
is somewhat counterintuitive in the light of the usual coupling constructions -- will be referred to as
{\em segregation} of two states. Further, we consider the constant
\begin{equation}\label{kappa}
\kappa:=\sup\lim_{n\to\infty}\lVert P^n(x,.)-P^n(y,.)\rVert_{\mathrm{TV}},
\end{equation}
where the supremum is taken over finite Markov chain transition matrices $P$ and states $x$ and $y$,
such that two copies of the chain corresponding to $P$, one started in $x$ and the other in $y$, can be
coupled to meet a.s.\ in finite time. To put it briefly, the main result of this paper is
that $\kappa$ equals $\tfrac12$.

As a preparation, the second section deals with the concept of couplings in general and convergence of
Markov chains. Much of this is standard, but there is also the lesser known but crucial distinction
between Markovian and faithful couplings. Section \ref{concept} presents H\"aggstr\"om's result 
($\kappa\geq3-2\sqrt{2}$) and puts the idea of segregating Markov chains into a broader context.

In Section \ref{upperbound}, more precisely in Theorem \ref{upperboundthm}, we prove that the value
$\tfrac12$ is an upper bound on $\kappa$.

In Section \ref{exampl}, a constructive and intuitively accessible example of a Markov chain is given,
that segregates two states such that the total variation distance kept can be pushed arbitrarily
close to $\tfrac{1}{\mathrm{e}}$. This improves on the example in \cite{disagreement} and serves as a
warmup for the more technical and implicit construction in the last section.

Finally, in Section \ref{ctg} we introduce and employ the idea of {\em separation} to show that for any
$\epsilon>0$, there exist Markov chains segregating two states $x$ and $y$ such that copies started in
these states can be coupled to meet almost surely while their distributions $P^n(x,.)$ and
$P^n(y,.)$ have a total variation distance of at least $\tfrac12-\epsilon$ for all $n\in\N$, see
Theorem \ref{lowerboundthm}. Together with the upper bound from Section \ref{upperbound}, this
establishes our main result, Theorem \ref{1/2}, stating that $\kappa=\tfrac12$.

\section{Preliminaries: convergence and couplings}

In order to quantify the difference between two probability measures (such as the distributions of
two copies of a Markov chain at a fixed time) there are quite a few distance measures. The so-called
total variation distance is among the most common ones.

\begin{definition}\label{TVdef}
	Let $\mu$ and $\nu$ be two probability distributions on a countable set $S$. The {\em total variation distance}
	between the two measures is then defined as
	$$\lVert\mu-\nu\rVert_{\mathrm{TV}}:=\sup_{A\subseteq S}|\mu(A)-\nu(A)|.$$
\end{definition}

This notion of distance is used in most of the standard convergence theorems on finite Markov chains as well
(e.g.\ see Thm.\ 4.9 in \cite{Mixing}):
\begin{theorem}\label{conv}
	Suppose $(X_n)_{n\in\N_0}$ is an irreducible and aperiodic Markov chain on a finite state space $S$.
	Then there exists a unique limiting distribution $\pi$ on $S$, called the stationary distribution,
	as well as constants $\alpha\in(0,1)$ and $C>0$ such that
	$$\lVert P^n(x,.)-\pi\rVert_{\mathrm{TV}}\leq C\,\alpha^n,$$ for all $x\in S,\ n>0$.
	%	where $P^n(x,.)$ denotes the distribution of $X_n$ given that the chain was started in $x$, i.e.\ $X_0=x$.	
\end{theorem}

If the distribution of a Markov chain at time $n$ converges to the same distribution $\pi$ as $n$ tends to 
infinity, irrespectively of its starting distribution, a standard way to measure the speed of convergence
is the variation distance
$$d(n):=\sup_{x\in S}\,\lVert P^n(x,.)-\pi\rVert_{\mathrm{TV}}.$$
Sometimes it is more convenient, however, to consider the related function
$$\overline{d}(n):=\sup_{x,y\in S}\,\lVert P^n(x,.)-P^n(y,.)\rVert_{\mathrm{TV}}.$$
Both these functions, $d$ and  $\overline{d}$, are non-increasing in $n$ and $\overline{d}$ is in addition
submultiplicative, i.e.\ $\overline{d}(m+n)\leq\overline{d}(m)\cdot\overline{d}(n)$. Submultiplicativity need
not hold for $d$, but can be verified for $2\,d$ instead. Furthermore, it holds that
$d(n)\leq \overline{d}(n)\leq 2\,d(n)$.
For proofs of the elementary facts just stated, we refer to Lemma 2.20 in \cite{Aldous}.
Note that there $S$ is assumed to be finite, but the arguments immediately transfer to countable $S$.

On the basis of the notion of distance $d$, the central concept of mixing time is defined, loosely speaking, as
the time it takes until the effect of the starting distribution has begun to disappear substantially.

\begin{definition}
	Given a Markov chain, for which the distribution of $X_n$ converges to a fixed distribution $\pi$ (irrespectively
	of the distribution of $X_0$), define its {\em mixing time} by
	$$t_\mathrm{mix}:=\min\{n\in\N_0;\ d(n)\leq\tfrac14\}.$$
\end{definition}

As already mentioned, the tool that often makes proofs about convergence of Markov chains both
short and elegant is the coupling approach. Let us therefore properly define this standard concept and
then highlight which additional properties a coupling can have.

\begin{definition}
	We define a {\em coupling} of two copies of a Markov chain on $S$ to be a process
	$((X_n,Y_n))_{n\in\N_0}$ on $S\times S$, with the property that both $(X_n)_{n\in\N_0}$ and
	$(Y_n)_{n\in\N_0}$ are Markov chains on $S$ with the same transition probabilities (but possibly
	different starting distributions).
	
	If the process $((X_n,Y_n))_{n\in\N_0}$ is itself a Markov chain (not necessarily time-homogeneous), it is called a {\em Markovian coupling}.
\end{definition}

%Couplings play an important role in the analysis of the long term behavior of Markov chains,
%especially their {\em mixing time}, loosely speaking the time it takes until the effect of the starting
%distribution starts to disappear.
In order to get good estimates on mixing times it is often of importance to bring into line the long term
behavior of the chain started in different states. In order to do so, one wants to make sure that the two
coupled chains stay together once they meet, more precisely: if $X_m=Y_m,$ then $X_n=Y_n,$ for all $n\geq m$.
Couplings with this property are sometimes called ``{\em sticky}'' couplings. As noted in the
introduction, it is however not possible to modify every coupling in such a way that it becomes sticky by
simply glueing together the two copies once they meet, see Prop.\ 3 in \cite{faithful} for an example.
The crucial property is the following:

\begin{definition}\label{faith}
	A Markovian coupling $((X_n,Y_n))_{n\in\N_0}$ of two copies of a Markov chain is called {\em faithful} if
	for all $x_n,y_n,x_{n+1},y_{n+1}\in S$, $n\in\N_0$:
	\begin{equation*}
	\begin{split}
	\Prob\big(X_{n+1}=x_{n+1}\,|\,(X_n,Y_n)=(x_n,y_n)\big)&=\Prob(X_{n+1}=x_{n+1}\,|\,X_n=x_n)\\
	&=P(x_n,x_{n+1})
	\end{split}
	\end{equation*}
	and
	\begin{equation*}
	\begin{split}
	\Prob\big(Y_{n+1}=y_{n+1}\,|\,(X_n,Y_n)=(x_n,y_n)\big)&=\Prob(Y_{n+1}=y_{n+1}\,|\,Y_n=y_n)\\
	&=P(y_n,y_{n+1}).\end{split}
	\end{equation*}
\end{definition}

It should be mentioned that the term ``Markovian coupling'' is used in \cite{Mixing} to describe what we just
defined as faithful coupling. However, since we actually want to focus on couplings that are not faithful
(but may still be Markov chains -- as both the example in Section \ref{exampl} and the one in
\cite{disagreement} are), we want to make this distinction by adopting the definitions in \cite{faithful} and 
deviate from the notions in \cite{Mixing}.

In order to understand what makes faithful couplings special, note that in general a coupling of two copies
${\mathbf X}$ and ${\mathbf Y}$ of a Markov chain with transition probabilities $\{P(r,s);\ r,s\in S\}$ fulfills
\begin{equation*}
\begin{split}
\sum_{y_n\in S}\Prob\big(X_{n+1}=x_{n+1}\,|\,(X_n,Y_n)=(x_n,y_n)\big)\cdot\Prob(Y_n=y_n\,|\,X_n=x_n)\\
=P(x_n,x_{n+1}),
\end{split}
\end{equation*}
for all $x_n,x_{n+1}\in S$, $n\in\N_0$, and likewise
\begin{equation*}
\begin{split}
\sum_{x_n\in S}\Prob\big(Y_{n+1}=y_{n+1}\,|\,(X_n,Y_n)=(x_n,y_n)\big)\cdot\Prob(X_n=x_n\,|\,Y_n=y_n)\\
=P(y_n,y_{n+1}),
\end{split}
\end{equation*}
for all $y_n,y_{n+1}\in S$, $n\in\N_0$. So the extra condition on a faithful coupling amounts to $\Prob\big(X_{n+1}=x_{n+1}\,|\,(X_n,Y_n)=(x_n,y_n)\big)$
being constant in $y_n$ and $\Prob\big(Y_{n+1}=y_{n+1}\,|\,(X_n,Y_n)=(x_n,y_n)\big)$ being constant in $x_n$.
\vspace*{1em}

It is immediate to check that any faithful coupling can indeed be transformed into a sticky coupling by just
letting the chains run according to the given coupling until they meet and then run them together as
two {\em identical} copies of the same chain, without affecting the marginals. Exploiting this fact leads to
the estimate
\begin{equation}\label{couplineq}
\lVert P^n(x,.)-P^n(y,.)\rVert_{\mathrm{TV}}\leq \Prob(\tau>n)=1-\Prob(\tau\leq n)
\end{equation}
for any faithful coupling of two copies, ${\mathbf X}$ started in $x$ and ${\mathbf Y}$ started in $y$, where
$$\tau:=\inf\{n\geq0;\ X_n=Y_n\}$$
is the {\em first meeting time} of the coupled chains (cf.\ Thm.\ 1 in \cite{faithful}).

\section{Chains that meet and separate}\label{concept}

If two copies of a Markov chain are coupled, but the coupling is not sticky, clearly they can meet in one
state and separate afterwards. As mentioned above, if the coupling is not faithful (i.e.\ violates the
conditions given in Definition \ref{faith}), in some cases it cannot be transformed into a sticky coupling
by simply letting the two copies coalesce once they meet. As a byproduct, H\"aggstr\"om \cite{disagreement}
observed an even stronger form of incompatibility of two coupled copies of a chain that meet. He gives an
example of a finite reducible Markov chain with the following property: Two copies of the chain, started in
different states $x$ and $y$, can be coupled in such a way that they meet almost surely in finite time,
while the total variation distance of their distributions never drops below a fixed positive value. More
precisely, he shows (see Prop.\ 4.1 in \cite{disagreement}):

\begin{proposition}\label{segregate}
	There exists a finite state Markov chain such that for two of its states $x$ and $y$ we have that
	$$\lim_{n\to\infty}\lVert P^n(x,.)-P^n(y,.)\rVert_{\mathrm{TV}}>0,$$
	while on the other hand there exists a Markovian coupling of the chains ${\mathbf X}=(X_n)_{n\in\N_0}$ and
	${\mathbf Y}=(Y_n)_{n\in\N_0}$,
	starting at $X_0=x$ and $Y_0=y$, with the property that their first meeting time 
	$\tau=\inf\{n\geq0;\ X_n=Y_n\}$ is finite with probability $1$.
\end{proposition}

Note that for any Markov chain and any two states $x$ and $y$, the sequence
$(\lVert P^n(x,.)-P^n(y,.)\rVert_{\mathrm{TV}})_{n\in\N_0}$ is non-increasing. This, together with the fact
that the total variation distance is always non-negative, guarantees the existence of $\lim_{n\to\infty}\lVert P^n(x,.)-P^n(y,.)\rVert_{\mathrm{TV}}$.

In fact, the reducible Markov chain in the example given in \cite{disagreement} comprises only $6$ states
(see Figure \ref{Ollesex} below). For $p\in[1-\frac{1}{2}\sqrt{2},\frac{1}{2}\sqrt{2}]$, two copies started
in $x$ and $y$ can be coupled such that their first meeting time is a.s.\ less than or equal to $2$ (for the
explicit calculations, see Prop.\ 4.1 in \cite{disagreement}). The copies
will reach one of the two absorbing states ($a$ and $b$) after two steps and the probability that the chain
started in $x$ lands in $a$ is $1-2p\,(1-p)$, in $b$ accordingly $2p\,(1-p)$. By symmetry, for the chain
started in $y$ it is precisely reversed.

\begin{figure}[H]
	\centering
	\includegraphics[scale=1]{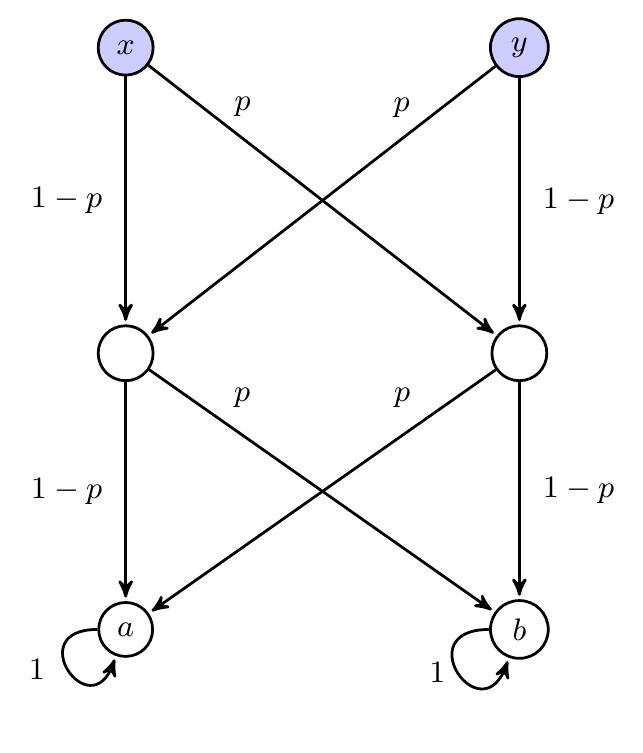}
	\caption{A first example of a segregating Markov chain -- provided $p\in[1-\frac{1}{2}\sqrt{2},\frac{1}{2}\sqrt{2}]\setminus\{\frac12\}$.}\label{Ollesex}
\end{figure}

So for $n\geq2$, $P^n(x,.)$ and $P^n(y,.)$ are unchanging and different if $p\neq\frac12$.
Choosing $p=\frac{1}{2}\sqrt{2}$ (or $p=1-\frac{1}{2}\sqrt{2}$) maximizes their total variation
distance at $3-2\sqrt{2}\approx 0.17153$.

As mentioned in the introduction, we will call Markov chains that have the property described in Proposition
\ref{segregate} to be {\em segregating the states $x$ and $y$}.
From the convergence theorem we know that such things cannot happen for {\em irreducible} finite Markov chains
(not even periodic ones).
%In the aperiodic case both $P^n(x,.)$ and $P^n(y,.)$ converge to the limiting distribution $\pi$,
%hence their total variation distance tends to $0$. In the periodic case, the fact that communicating states
%have the same period forces all states to have period $d$ for some $d\geq2$. One can then define an equivalence
%relation by
%$$x\sim y\quad \Leftrightarrow \quad\exists\ n\in\N_0\text{\ \ such that\ \ } \Prob(X_{nd}=y\,|\,X_0=x)>0,$$
%that divides the state space $S$ into $d$ equivalence classes. The thinned chain $\{X_{nd}\}_{n\in\N_0}$ started
%in $x$ becomes an irreducible and aperiodic Markov chain on the equivalence class $[x]$. If $y\in[x]$,
%$\lim_{n\to\infty}\lVert P^n(x,.)-P^n(y,.)\rVert_{\mathrm{TV}}=0$ is again forced by Theorem \ref{conv}, this
%time applied on the thinned chains.
%If $y\notin[x]$, the chains started in $x$ and $y$ will never meet, making $\Prob(\tau<\infty)=1$ impossible.
%
In addition to that, even if the chain is either reducible or infinite, a coupling that lets two copies
started in different states meet almost surely, while the total variation distance of their distributions
stays bounded away from 0 for all time, cannot be faithful due to the coupling inequality \eqref{couplineq}.
%: If it was, the modification to a sticky coupling would change neither the first meeting time
%$\tau$ nor $P^n(x,.)$ nor $P^n(y,.)$ and the condition $\Prob(\tau<\infty)=1$ would then imply 
%$$\lim_{n\to\infty}\lVert P^n(x,.)-P^n(y,.)\rVert_{\mathrm{TV}}=0,$$
%as a sticky coupling of ${\mathbf X}$ and ${\mathbf Y}$ enforces $\{\tau\leq n\}\subseteq\{X_n=Y_n\}$.

\section{The upper bound}\label{upperbound}

From the previous section we know that there exist finite reducible Markov chains that 
{\em segregate two states $x$ and $y$}.
% i.e.\ one can couple two copies started in $x$ and $y$ respectively 
% in such a way that they meet almost surely in finite time while $P^n(x,.)$ and $P^n(y,.)$ maintain a total
% variation distance bounded away from 0 as $n$ tends to infinity.
A natural question in this respect is how large a total variation distance between $P^n(x,.)$ and $P^n(y,.)$
can be kept, under the condition that two copies started in $x$ and $y$ respectively can be coupled to meet in
finite time almost surely -- in other words, the value of $\kappa$ as defined in \eqref{kappa}. The example in
\cite{disagreement} shows $\kappa \geq 3-2\sqrt{2}$; the following theorem establishes $\tfrac12$ as an upper
bound.

\begin{theorem}\label{upperboundthm}
	Consider a Markov chain on the countable state space $S$ and two fixed states $x$ and $y$. Further, we
	denote by ${\mathbf X}=(X_n)_{n\in\N_0}$ and ${\mathbf Y}=(Y_n)_{n\in\N_0}$ two coupled copies of the
	chain, started in $x$ and $y$ respectively,
	and their first meeting time by $\tau$. If ${\mathbf X}$ and ${\mathbf Y}$ can be coupled in such a way that $\Prob(\tau<\infty)=1$,
	it holds that
	$$\lim_{n\to\infty}\lVert P^n(x,.)-P^n(y,.)\rVert_{\mathrm{TV}}\leq \tfrac12.$$
\end{theorem}

This result is an immediate implication of the following proposition:

\begin{proposition}\label{coupl}
	Consider ${\mathbf X}=(X_n)_{n\in\N_0}$, ${\mathbf Y}=(Y_n)_{n\in\N_0}$ and $\tau$ as above. Then it follows that
	$$\lVert P^n(x,.)-P^n(y,.)\rVert_{\mathrm{TV}}\leq 1-\tfrac12\cdot\Prob(\tau\leq n)\quad \text{for all }
	n\in\N_0.$$
\end{proposition}

\begin{proof}
	Fix $n\in \mathbb{N}_0$ and a subset $A \subseteq S$ by means of which we define the processes $(M_t)_{t=0}^n$ and $(N_t)_{t=0}^n$ given by
	$$M_t := \mathbb{P}\left( X_n \in A \,|\, X_t\right) = P^{n-t}(X_t, A),$$ and
	$$N_t := \mathbb{P}\left( Y_n \in A \,|\, Y_t\right) = P^{n-t}(Y_t, A).$$
	It is easily checked that these processes are martingales (with respect to the filtrations generated by
	$\mathbf{X}$ and $\mathbf{Y}$ respectively). Further let $B_x$ and $B_y$ denote the events that
	$M_t \geq \frac{1}{2}$ and $N_t < \frac{1}{2}$ respectively for all $0 \leq t \leq n$. As the event
	$B_x\cap B_y$ implies $M_t\neq N_t$ and with that $X_t\neq Y_t$ for all $0 \leq t \leq n$ (almost surely), it follows
	that $\{\tau\leq n\}$ is (up to a nullset) contained in the union of $B_x^{\,\text{c}}$ and $B_y^{\,\text{c}}$.
	
	Next, we define
	$$\tau_x := \inf\{0 \leq t \leq n;\ M_t < \tfrac{1}{2}\}\quad\text{and}\quad
	\tau_y := \inf\{0 \leq t \leq n;\ N_t \geq \tfrac{1}{2}\},$$
	where the infimum is understood to be $n$ if the corresponding set is empty. Note that $\tau_x$ and $\tau_y$ are stopping times for $M_t$ and $N_t$ respectively. Since $(M_t)_{t=0}^n$ and $(N_t)_{t=0}^n$ are bounded martingales,
	the Optional Stopping Theorem (see for example Cor.\ 17.7 in \cite{Mixing}) gives the estimates
	\begin{equation}\label{M_t} P^n(x, A) = \mathbb{E}\,M_0 = \mathbb{E}\,M_{\tau_x} \leq \tfrac{1}{2} \cdot \mathbb{P}(B_x^{\,\text{c}}) 
	+ 1 \cdot \mathbb{P}(B_x) = 1 - \tfrac{1}{2} \cdot \mathbb{P}(B_x^{\,\text{c}})\end{equation}
	$$ \text{and}\quad P^n(y, A) = \mathbb{E}\,N_0 = \mathbb{E}\,N_{\tau_y} \geq \tfrac{1}{2}\cdot\mathbb{P}(B_y^{\,\text{c}}).$$
	Combining these two inequalities, we get
	$$P^n(x, A) - P^n(y, A) \leq 1 - \tfrac{1}{2}\left( \mathbb{P}(B_x^{\,\text{c}}) + \mathbb{P}(B_y^{\,\text{c}})\right) \leq 1 - \tfrac{1}{2}\cdot \mathbb{P}(\tau\leq n).$$
	Finally, maximizing the left-hand side over all subsets $A\subseteq S$ yields
	$$\|P^n(x, \cdot) - P^n(y, \cdot)\|_{TV} \leq 1 - \tfrac{1}{2}\cdot\mathbb{P}(\tau\leq n),$$
	as claimed.
\end{proof}

\begin{remark}\label{general}
	Reading carefully through the proof of Proposition \ref{coupl}, one may notice that the martingale argument
	used essentially does not require our general assumptions of time-homogeneity and countable state space.
	For time-inhomogeneous chains, $M_t:=\mathbb{P}\left( X_n \in A \,|\, X_t\right)$, can no longer
	be written as $P^{n-t}(X_t, A)$ (likewise for $N_t$), but this does not impair the argument and we can again
	conclude $\lim_{n\to\infty}\lVert \mathcal{L}(X_n)-\mathcal{L}(Y_n)\rVert_{\mathrm{TV}}\leq \tfrac12$, where
	$\mathcal{L}(X_n)$ denotes the distribution of $X_n$. Given an	uncountable state space, the first meeting
	time $\tau$ is no longer measurable by default. If we add this as an extra condition, however, the above
	proof (with the minor modification that only measurable sets $A$ are considered) extends to this setting
	as well.
\end{remark}

To fully exhaust the range of validity of this argument, let us leave the default preconditions for
a moment and consider continuous-time Markov processes in full generality (as e.g.\ Kallenberg
\cite{Kallenberg} defines them). In this case, we have to add further technical assumptions to save
the line of reasoning and result of Proposition \ref{coupl}. For a measurable set $A\subseteq S$, time
horizon $T>0$ and
$0\leq t\leq T$, define (similar to the above)
\begin{equation}\label{martingale}
M_t:=\mathbb{P}\left( X_T \in A \,|\, X_t\right) \quad \text{and}\quad
N_t:=\mathbb{P}\left( Y_T \in A \,|\, Y_t\right).
\end{equation}

\begin{theorem}\label{generalthm}
	Consider a continuous-time Markov process (not necessarily time-ho\-mo\-ge\-neous) with general
	state space $S$. Let ${\mathbf X}=(X_t)_{t\geq0}$ and ${\mathbf Y}=(Y_t)_{t\geq0}$ denote two coupled
	copies of the process, that are started in fixed states $x$ and $y$ respectively, and let $\tau$ denote
	their first meeting time. Fix a time horizon $T>0$ and assume that $\{\tau\leq T\}$ is measurable.
	If for all measurable sets $A\subseteq S$ it is possible to choose versions of the martingales
	$(M_t)_{t\in[0,T]}$ and $(N_t)_{t\in[0,T]}$, as defined in \eqref{martingale}, that are a.s.\ continuous from the right, while having the
	property that for all $t\in[0,T]$, $X_t=Y_t$ implies $M_t=N_t$, it holds that
	$$\lVert \mathcal{L}(X_T)-\mathcal{L}(Y_T)\rVert_{\mathrm{TV}}\leq 1-\tfrac12\cdot\Prob(\tau\leq T).$$
\end{theorem}

\begin{proof}
	Again, we fix some measurable subset $A\subseteq S$ and define the martingales $(M_t)_{t\in[0,T]}$ and
	$(N_t)_{t\in[0,T]}$ as in \eqref{martingale}, with the two additional properties stated in the theorem.
	Following the proof of Proposition \ref{coupl} (literally, besides replacing $n$ by $T$) we can still
	conclude that $\{\tau\leq T\}\cap(B_x\cap B_y)$ is a nullset. Continuity from the right of $(M_t)_{t\in[0,T]}$
	and	$(N_t)_{t\in[0,T]}$ implies $M_{\tau_x}\leq\frac12$ on $B_x^{\,\text{c}}$ and $N_{\tau_y}\geq\frac12$ on $B_y^{\,\text{c}}$. Using the Optional Stopping Theorem for continuous-time martingales (see e.g.
	Thm.\ (3.2) in \cite{Revuz}) we can conclude
	just as above.
\end{proof}

\begin{remark}\label{general2}
	%the trajectories of $M_t$ and $N_t$ must be continuous from the right in order for us to be able to apply the standard continuous-time Optional Stopping Theorem.
	A simple way to ensure the assumed properties of $(M_t)_{t\in[0,T]}$ and
	$(N_t)_{t\in[0,T]}$ is to consider a topology on $S$ and to require two things: first, that the Markov process a.s.\ has right-continuous sample paths and second that the transition probabilities $\mathbb{P}(X_T\in A \,|\, X_t=x)$ are continuous in $t\in[0, T)$ and $x\in S$ for all measurable $A\subseteq S$.
	
	As an aside, it might be worth noting that the analogue of Proposition \ref{coupl} is not valid if we drop the additional assumptions completely, as it might be possible then to alter trajectories without changing
	transition probabilities. For instance, consider the process $(X_t)_{t\in[0,1]}$ on $S=\{0, 1\}$ defined by
	$$X_t=\begin{cases}X_0,& \text{for } t\neq \xi \\
	1-X_0,& \text{for } t=\xi,\end{cases}
	\quad\text{where } \xi \sim \mathrm{unif}(0,1).$$
	Two independent copies, started at $0$ and $1$ respectively and using independent copies of $\xi$, will almost surely meet, but the total variation distance of their distributions stays $1$ for all $t\in[0,1]$.
\end{remark}

Besides these generalizations, the statement from Proposition \ref{coupl} can also be used to get
upper bounds on mixing times -- similar to the usual approach, see for example Cor.\ 5.3 in \cite{Mixing}
-- replacing the coupling inequality \eqref{couplineq} as starting point. In doing so, we pay by an
additional factor $\tfrac12$ in front of $\Prob(\tau\leq n)$, but can in return employ any kind of coupling,
not only faithful ones.
It remains to be seen whether this will ever turn out useful in practice as basically all
standard coupling constructions are faithful. However, we want to mention at this point that
non-Markovian couplings actually already proved to be useful in applications, cf.\ \cite{Vigoda}
for instance.

\begin{proposition}\label{tmix}
	Consider a Markov chain $\mathbf{X}=(X_n)_{n\in\N_0}$ with the property that $\mathcal{L}(X_n)$ converges to a fixed distribution $\pi$ irrespectively of the starting distribution.
	Further, suppose that for some $\alpha\in(0,1]$ and each pair of states $x,y\in S$ there exists a (not
	necessarily faithful or even Markovian) coupling $((X_n,Y_n))_{n\in\N_0}$ of two copies of the
	chain started in $x$ and $y$ respectively, such that the first meeting time $\tau$ of the
	two coupled processes fulfills $\Prob(\tau\leq n)\geq \alpha$. Then
	$$t_\mathrm{mix}\leq n\cdot\left\lceil\frac{\log(\tfrac14)}{\log(1-\tfrac{\alpha}{2})}\right\rceil.$$
\end{proposition}

\begin{proof}
	From Proposition \ref{coupl} we can conclude that 
	$$\overline{d}(n)=\sup_{x,y\in S}\,\lVert P^n(x,.)-P^n(y,.)\rVert_{\mathrm{TV}}\leq1-\tfrac{\alpha}{2}.$$
	Consequently, as $\overline{d}$ is submultiplicative and dominates $d$, we get for any $k\in\N:$
	$$d(kn)\leq\overline{d}(kn)\leq\overline{d}(n)^k\leq(1-\tfrac{\alpha}{2})^k.$$
	Choosing $k\geq\frac{\log\left(\tfrac14\right)}{\log\left(1-\tfrac{\alpha}{2}\right)}$, this estimate
	becomes
	$$d(kn)\leq(1-\tfrac{\alpha}{2})^{\log\left(\tfrac14\right)\big/\log\left(1-\tfrac{\alpha}{2}\right)}
	=\frac14.\vspace*{-1em}$$
\end{proof}

\section{A simple example that narrows the gap}\label{exampl}

In this section, we will present another finite state Markov chain that improves on the value of $3-2\sqrt{2}$
established by Häggström \cite{disagreement}. To begin with, let us prepare a lemma, which will come in useful
when the total variation distance in our example of a finite segregating Markov chain is to be assessed.

Consider a sequence of independent Bernoulli trials, each with success probability $p<1$. The distribution of
the number of successful attempts until $r$ failures have occurred is called the negative binomial distribution
with parameters $r$ and $p$ and commonly denoted by $\mathrm{NB}(r,p)$.

\begin{lemma}\label{TV}
	For $\mu:=\mathrm{NB}(1,p)$ and $\nu:=\mathrm{NB}(2,p)$, it holds that
	\begin{equation}\label{NB}
	\lVert \mu -\nu\rVert_{\mathrm{TV}}
	=\sum_{k=0}^{\big\lfloor\tfrac{p}{1-p}\big\rfloor} \big(\mu(k)-\nu(k)\big)
	=\big(\big\lfloor\tfrac{p}{1-p}\big\rfloor+1\big)
	\cdot(1-p)\cdot p^{\big\lfloor\tfrac{p}{1-p}\big\rfloor+1}
	\end{equation}
	for any $p\in[0,1)$. If $p=\tfrac{m}{m+1}$, where $m\in\N_0$, this value simplifies to $p^{\tfrac{1}{1-p}}$.
\end{lemma}

\begin{proof}
	A standard calculation shows that for two probability distributions $\mu$ and $\nu$ on a discrete space $S$,
	their total variation distance can be calculated as
	$$\lVert\mu-\nu\rVert_{\mathrm{TV}}=\sum\limits_{\substack{x\in S\\
			\mu(x)\geq\nu(x)}}\big(\mu(x)-\nu(x)\big),$$
	see for example Remark 4.3 in \cite{Mixing}. For $\mu=\mathrm{NB}(1,p)$ and $\nu=\mathrm{NB}(2,p)$, we
	have:
	$$\mu(k)=p^k\,(1-p)\quad\text{and}\quad\nu(k)=(k+1)\,p^k\,(1-p)^2,\quad \text{for all } k\in\N_0,$$
	which implies $\mu(k)\geq\nu(k)$ for $k\leq\big\lfloor\tfrac{p}{1-p}\big\rfloor=:m$.
	
	Consequently, using	the well-known formulas for a finite geometric sum and its derivative, we can compute
	the total variation distance and get
	\begin{align*}
	\lVert\mu-\nu\rVert_{\mathrm{TV}}&=\sum_{k=0}^m p^k\,(1-p)-\sum_{k=0}^m (k+1)\,p^k\,(1-p)^2\\
	&=1-p^{m+1}-\big[1-p^{m+2}-(m+2)\,p^{m+1}\,(1-p)\big]\\
	&=(m+1)\,(1-p)\,p^{m+1}.
	\end{align*}
	If $p=\tfrac{m}{m+1}$, for some $m\in\N_0$, the number $\tfrac{p}{1-p}=m$ is integer and an elementary
	simplification of the expression to the right in \eqref{NB} verifies the final claim.
\end{proof}
\vspace*{1em}

Let us now use this lemma to establish the following result:

\begin{proposition}\label{ex}
	For all $\epsilon>0$, there exists a Markov chain segregating two states $x$ and $y$
	such that
	$$\lim_{n\to\infty}\lVert P^n(x,.)-P^n(y,.)\rVert_{\mathrm{TV}}>\frac{1}{\mathrm{e}}-\epsilon.$$
\end{proposition}

\begin{proof}
	Let us consider the finite reducible MC depicted in Figure \ref{segex}. The state space $S$
	comprises $3m+5$ states, among which the two initial states $x$ and $y$ as well as the $m+2$ absorbing
	states labeled $0,\dots,m$ and $>$. 
	
	\begin{figure}[H]
		\centering
		\includegraphics[scale=1]{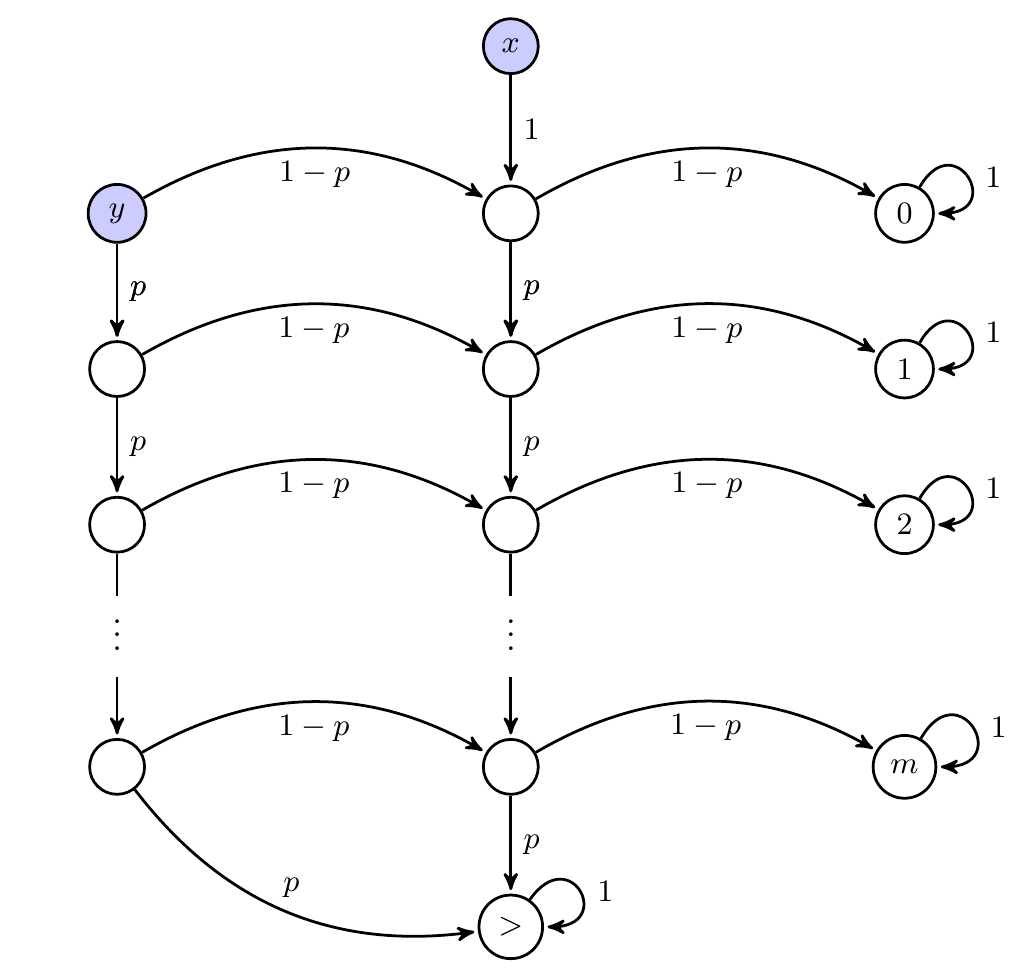}
		\caption{A segregating MC allowing for a fairly large total variation distance of two
			copies started in $x$ and $y$ respectively.}\label{segex}
	\end{figure}
	It is immediate to check that the two copies, started in $x$ respectively $y$, will hit an absorbing
	state after at most $m+2$ steps and that on $\{0,\dots,m\}$ the distributions $P^{m+2}(x,.)$ and $P^{m+2}(y,.)$ coincide
	with $\mathrm{NB}(1,p)$ and $\mathrm{NB}(2,p)$ respectively. The probabilities to land in the state
	labeled $>$ are
	$$P^{m+2}(x,>)=\Prob(Z_1> m)\quad\text{and}\quad P^{m+2}(y,>)=\Prob(Z_2> m),$$ 
	where $Z_1\sim\mathrm{NB}(1,p)$ and $Z_2\sim\mathrm{NB}(2,p)$.
	
	Choosing $p=\tfrac{m}{m+1}$, we can conclude from Lemma \ref{TV} that
	$$\lVert P^n(x,.)-P^n(y,.)\rVert_{\mathrm{TV}}=p^{\tfrac{1}{1-p}},$$ for all $n\geq m+2$. Writing
	$q:=\tfrac{1}{1-p}=m+1$ shows
	$$\lim_{p\to 1}p^{\tfrac{1}{1-p}}=\lim_{q\to\infty}\big(1-\tfrac1q\big)^q=\frac{1}{\mathrm{e}}.$$
	Taking $m$ large enough, more precisely such that $\big(1-\frac{1}{m+1}\big)^{m+1}>
	\frac{1}{\mathrm{e}}-\epsilon$, will establish the claim if we can present a coupling that ensures 
	that the two copies started in $x$ and $y$ will meet with probability $1$, either before or when they
	hit an absorbing state.
	
	In order to establish such a coupling, let ${\mathbf Y}=(Y_n)_{n\in\N_0}$ be a copy of the MC started in $y$. The
	copy ${\mathbf X}=(X_n)_{n\in\N_0}$ started in $x$ mimicks all movements of ${\mathbf Y}$, with the delay of one step, until
	it finally hits an absorbing state: First, it will move downwards until the two processes meet -- in
	particular this implies that its first step is downwards with probability 1, as $x\neq y$. Then,
	once $X_n=Y_n$ for some $1\leq n\leq m+1$, the next step of the process ${\mathbf X}$ is to move to the right to
	an absorbing state, i.e.\ $X_{n+1}=n-1$. If ${\mathbf Y}$ never moves to the right, neither does ${\mathbf X}$ and
	both finally end up in the state $>$.
	
	First of all, we need to check whether the two coordinate processes are indeed copies of the MC given in
	Figure \ref{segex}: It is obvious from our construction that it suffices to verify this for the process ${\mathbf X}$.
	The way ${\mathbf X}$ is defined -- to move downwards in the first step and then always imitate the previous move
	of ${\mathbf Y}$ until ending up in an absorbing state -- gives the right marginals due to the structure of the MC: As
	all the non-absorbing states apart from $x$ have the same transition probabilities ($p$ downwards and
	$1-p$ to the right), ${\mathbf X}$ performs indeed a random walk on the graph in Figure \ref{segex} according to the
	transition probabilites of the MC. Note that there is just this one way ${\mathbf Y}$ can end up in an absorbing state
	before it meets ${\mathbf X}$, namely if it moves downwards only. Then, however, ${\mathbf X}$ copies this
	behavior and ends up in the state $>$ as well, so the coupling guarantees $\tau\leq m+2$ with probability $1$.
	This trivially implies the almost sure finiteness of the first meeting time $\tau$ and in conclusion the claim
	that the MC segregates the
	two states $x$ and $y$.
\end{proof}
\vspace*{1em}

In order to get the idea of how faithfulness plays a crucial role in this context, it is worth noting
that the coupling in our example -- likewise the one in \cite{disagreement} -- is in fact Markovian, but
not faithful. Such couplings, however, have to be non-faithful as faithfulness would imply that the total
variation distance necessarily tends to $0$ as already mentioned at the end of Section \ref{concept}.

\section{Closing the gap}\label{ctg}

In this last section, we want to further improve the lower bound $\frac{1}{\mathrm{e}}$, established
by the example from the previous section, in order to determine the true value of the constant $\kappa$,
defined in \eqref{kappa}. These efforts amount to the following:

\begin{theorem}\label{1/2}
	The value of $\kappa$, denoting the supremum of
	$\lim_{n\to\infty}\lVert P^n(x,.)-P^n(y,.)\rVert_{\mathrm{TV}}$
	taken over all segregating Markov chains and segregated states $x$ and $y$, is $\tfrac12$.
\end{theorem}

In view of Theorem \ref{upperboundthm}, the final step to derive this result lies in proving that for any
$\epsilon>0$, a value of at least $\tfrac12-\epsilon$ can actually be attained. In order to do so, we will
focus on (reducible) Markov chains with a specific structure which allows us to consider even simpler chains
in finite time instead:
When talking about a Markov chain with finite time horizon ${\mathbf X}=(X_t)_{t=0}^T$ on state space $S$ with
transition probabilities $\{P(r,s);\ r,s\in S\}$, we are in fact thinking of a different chain, namely
the reducible Markov chain ${\mathbf Y}=(Y_n)_{n\in\N_0}$ on the state space $S\times\{0,\dots,T\}$, with
transition probabilities 
$$\Prob(Y_{n+1}=(s,n+1)\,|\,Y_n=(r,n))=P(r,s)\quad\text{for all } r,s\in S,\ n\in\{0,\dots,T-1\}$$
$$\text{and}\quad\Prob(Y_{n+1}=(s,T)\,|\,Y_n=(s,T))=1.$$
In other words, we consider the evolution in time as new generations
of states and stop the original chain at time $T$ by making all states corresponding to time $T$ absorbing.
Incidentally, the example given in \cite{disagreement} is also of this kind; it corresponds to a two state
Markov chain with $T=2$ (cf.\ Figure \ref{Ollesex}).

So for the remainder of this article, we actually consider Markov chains in discrete, finite time
and with finite state space only. With this simplification in mind, we want to prove the following:

\begin{theorem}\label{lowerboundthm}
	For any $\epsilon>0$ there exists a Markov chain $\mathbf{X}$, two states $x$ and $y$ and a positive
	integer $T$ such that
	$$\lVert P^T(x,.)- P^T(y,.)\rVert_{\mathrm{TV}} \geq \frac{1}{2}-\epsilon$$
	and such that there is a coupling of two copies of the chain, $((X_t,Y_t))_{t=0}^T$, with initial states $x$ and
	$y$ respectively satisfying
	$$\mathbb{P}\left(X_t =Y_t \text{ \upshape for some }0\leq t \leq T\right) = 1.$$
\end{theorem}

Before proceeding with a proof of this theorem, we want to introduce an alternative way to view
segregating couplings. Let $\mathbf{X}$ be any discrete-time Markov chain with a finite state space $S$,
and fix two states $x, y\in S$ and a positive integer $T$.

For any coupling of two copies of this chain, $((X_t,Y_t))_{t=0}^T$, started in $X_0=x$ and $Y_0=y$
respectively, one can consider the corresponding \emph{meeting probability}
\begin{equation}\label{meetprob}
\mathbb{P}\left(X_t =Y_t\text{ for some }0\leq t\leq T\right).
\end{equation}
A natural question to ask in this setting is how large one can make this probability for a given (finite) Markov
chain and given $x, y$ and $T$ by maximizing over all such couplings. Let $\mathcal{X}$ denote the subset
of $\{x\}\times S^{T}$ consisting of all possible trajectories of $\mathbf{X}=(X_t)_{t=0}^T$, and
respectively $\mathcal{Y}\subseteq \{y\}\times S^{T}$ for $\mathbf{Y}=(Y_t)_{t=0}^T$.
Observe that any coupling of $\mathbf{X}$ and $\mathbf{Y}$ is determined by the values
$(p_{\mathbf{xy}})_{\mathbf{x}\in\mathcal{X}, \mathbf{y}\in\mathcal{Y}}$, where $p_{\mathbf{xy}}$ denotes
the probability of the event that both $\mathbf{X}=\mathbf{x}=(x_t)_{t=0}^T$ and $\mathbf{Y}=\mathbf{y}$.
We will denote the maximal value of \eqref{meetprob} by $\mathcal{C}_T(x,y)$ and call it the
\emph{optimal meeting probability}. 

%The notation $\mathbf{x}\sim\mathbf{y}$ in this context stands for two trajectories that share
%at least one state, i.e.\ $x_t=y_t$ for some $0\leq t \leq T$.

While finding explicit couplings that maximize the meeting probability can quickly become cumbersome as the
number of possible trajectories grows, it turns out that the problem of optimizing the meeting probability
has a useful dual, which allows us to determine $\mathcal{C}_T(x,y)$ without having to deal with the
couplings directly. This duality corresponds to the idea of max-flow/min-cut and K\"onig's theorem in
combinatorial optimization.

\begin{definition}
	Let $\mathcal{A}=(A_t)_{t=0}^T$ be a sequence of subsets of $S$, the (finite) state space of the considered Markov
	chain $\mathbf{X}$. We will refer to any such sequence as a \emph{separating sequence}. We define the
	\emph{separation} of any separating sequence as
	\begin{equation}\label{eq:separation}
	\begin{split}
	\mathcal{S}^{\mathcal{A}}_T(x,y) &= \mathbb{P}\left( X_t \in A_t\text{ for all }0\leq t \leq T \,|\, X_0=x\right)\\ &\hspace{.1em}+ \mathbb{P}\left( X_t \not\in A_t\text{ for all }0\leq t \leq T \,|\,X_0=y\right).
	\end{split}
	\end{equation}
	We say that the separating sequence is \emph{non-trivial} if both summands on the right-hand side in
	\eqref{eq:separation} are non-zero. We further define the \emph{optimal separation} $\mathcal{S}_T(x,y)$
	as the maximum separation over all possible separating sequences.
\end{definition}

It is not too hard to see that the optimal meeting probability and optimal separation are related.
Specifically, for any coupling $((X_t, Y_t))_{t=0}^T$ such that $(X_0, Y_0) = (x, y)$ and any
separating sequence $\mathcal{A}=(A_t)_{t=0}^T$ we have
\begin{equation*}
\begin{split}
&\mathbb{P}\left(X_t =Y_t\text{ for some }0\leq t\leq T\right)\\
&\qquad\leq \mathbb{P}\left( X_t \not\in A_t \text{ or } Y_t \in A_t \text{ for some }0\leq t\leq T\right)\\
&\qquad\leq \mathbb{P}\left( X_t \not\in A_t\text{ for some }0\leq t\leq T\right)
+\mathbb{P}\left( Y_t \in A_t\text{ for some }0\leq t\leq T\right)\\
&\qquad= 2-\mathcal{S}^{\mathcal{A}}_T(x,y).
\end{split}
\end{equation*}
Maximizing the meeting probability over all possible couplings of two copies, started in $x,y$
respectively, and minimizing the upper bound by maximizing the separation $\mathcal{S}^{\mathcal{A}}_T(x,y)$
over all separating sequences yields
\begin{equation}\label{eq:oneway}
\mathcal{C}_T(x,y)\leq 2-\mathcal{S}_T(x,y).
\end{equation}
However, for our purposes (namely to guarantee $\mathcal{C}_T(x,y)$=1), we rather need to bound
$\mathcal{C}_T(x,y)$ from below.
In this respect it is quite convenient that the inequality \eqref{eq:oneway} actually holds as an equality, as
the following theorem shows.
\begin{theorem}\label{thm:strongduality}
	Given an arbitrary but fixed finite Markov chain $\mathbf{X}$, two states $x,y\in S$ and
	time horizon $T$, we have
	\begin{equation}\label{eq:strongduality}
	\mathcal{C}_T(x,y) = 2-\mathcal{S}_T(x,y).
	\end{equation}
\end{theorem}

\begin{proof}
	A simple way to prove the reverse inequality is to employ the max-flow min-cut theorem, in the same
	way it can be used to prove Strassen's monotone coupling theorem. Starting from the sets $\mathcal{X}$
	and $\mathcal{Y}$ as above, we build the following directed graph, which will be denoted by $\oldvec{G}=(V,\oldvec{E})$:
	
	First we let each $\mathbf{x}\in\mathcal{X}$ and $\mathbf{y}\in\mathcal{Y}$ be represented by a node.
	Then we add two further nodes: a source $s$ and a sink $t$. When it comes to the directed edges, there
	will be an arrow $(s,\mathbf{x})$ for all $\mathbf{x}\in\mathcal{X}$ and $(\mathbf{y},t)$ for all 
	$\mathbf{y}\in\mathcal{Y}$. Additionally, we include the edge $(\mathbf{x},\mathbf{y})$, if the two
	trajectories $\mathbf{x}\in\mathcal{X}$ and $\mathbf{y}\in\mathcal{Y}$ share at least one state,
	i.e.\ $x_t=y_t$ for some $0\leq t \leq T$; in the sequel, we will write this as $\mathbf{x}\sim\mathbf{y}$.
	
	Finally, we have to assign capacities to these directed edges: The edges $(s,\mathbf{x})$ and $(\mathbf{y},t)$
	will get capacities $\Prob(\mathbf{X}=\mathbf{x})$ and $\Prob(\mathbf{Y}=\mathbf{y})$ respectively.
	All edges going in between $\mathcal{X}$ and $\mathcal{Y}$ get capacity $1$, see Figure \ref{mcmfgraph}
	below for an illustration.
	
	\begin{figure}[H]
		\centering
		\includegraphics[width=0.9\textwidth]{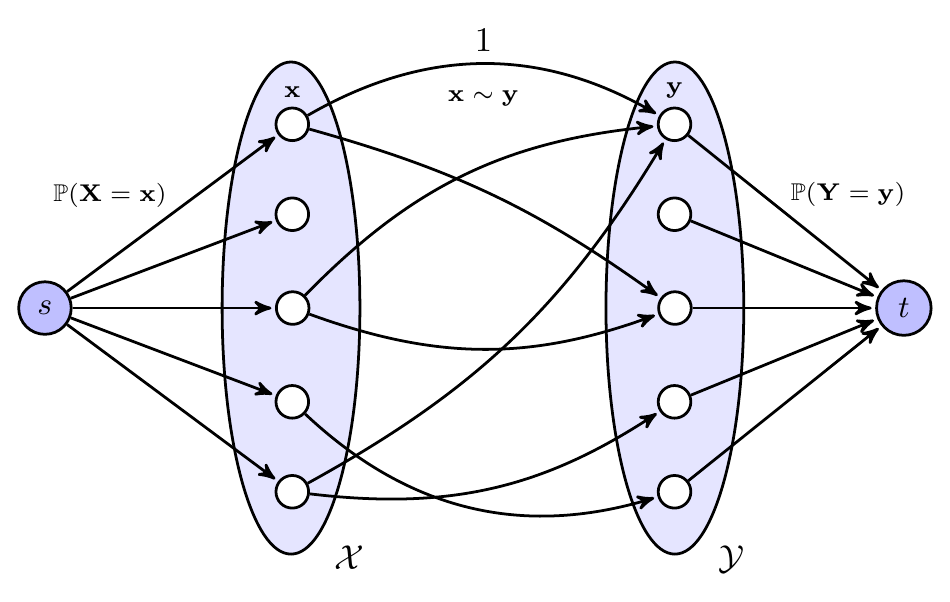}
		\caption{The auxiliary graph $\oldvec{G}$ to which we apply the max-flow min-cut theorem.}\label{mcmfgraph}
	\end{figure}
	
	Let us now consider the minimum cut problem on $\oldvec{G}$.
	From the fact that the cut $(\{s\},V\setminus\{s\})$ has value $1$, we know that we can focus on cuts
	that are not cutting any edges going in between $\mathcal{X}$ and $\mathcal{Y}$ when trying to find a minimal
	one. For a cut $(B,B^\text{c})$ of this kind, with $s\in B$ and $t\in B^\text{c}$ say, we note that
	$\mathbf{x}\sim\mathbf{y}$ can not occur for $\mathbf{x}\in\mathcal{X}\cap B$ and
	$\mathbf{y}\in\mathcal{Y}\cap B^\text{c}$, due to the assumption that no such edges are cut.
	Furthermore, the edges cut incident to $s$
	have at least the value $\Prob(X\in \mathcal{X}\cap B^\text{c})$, the ones incident to $t$ at least
	$\Prob(Y\in \mathcal{Y}\cap B)$.
	
	Let us define the set sequence $\mathcal{A}=(A_t)_{t=0}^T$ by
	$$A_t:=\{x_t;\ \mathbf{x}=(x_s)_{s=0}^T\in \mathcal{X}\cap B\},\text{ for all } 0\leq t\leq T,$$
	in order to bound the value of the given cut from below by
	\begin{align*}\Prob(X\in \mathcal{X}\cap B^\text{c})+\Prob(Y\in \mathcal{Y}\cap B)\\
	&\hspace{-4cm}\geq 1-\mathbb{P}(X_t \in A_t\text{ for all }0\leq t \leq T)
	+ 1-\mathbb{P}(Y_t \not\in A_t\text{ for all }0\leq t \leq T)\\
	&\hspace{-4cm}= 2-\mathcal{S}^{\mathcal{A}}_T(x,y)\\
	&\hspace{-4cm}\geq 2-\mathcal{S}_T(x,y).
	\end{align*}
	Consequently, as this bound applies to any minimal cut, using the max-flow min-cut theorem (see for example
	Thm.\ 1, Chap.\ III in \cite{Bollobas}) we are guaranteed the existence of a maximal flow of value at least
	$2-\mathcal{S}_T(x,y)$. Let us denote the respective flow through the edge corresponding to
	$\mathbf{x}\sim\mathbf{y}$ by $q_{\mathbf{x}\mathbf{y}}$.
	
	We can use this maximal flow to establish a coupling of $\mathbf{X}$ and $\mathbf{Y}$ in the same vein
	as in Doeblin's coupling lemma (see for instance Prop.\ 4.7 in \cite{Mixing}):
	First we let $\mathbf{X}=\mathbf{x}$ and $\mathbf{Y}=\mathbf{y}$ simultaneously with probability
	$q_{\mathbf{x}\mathbf{y}}$ for all $\mathbf{x}\sim\mathbf{y}$.
	Then we define $\mathbf{X}$ to follow the trajectory $\mathbf{x}$ with the remaining probability
	$$\Prob(\mathbf{X}=\mathbf{x})-\sum_{\mathbf{y}:\ \mathbf{x}\sim\mathbf{y}} q_{\mathbf{x}\mathbf{y}}$$
	and similarly $\mathbf{Y}$ to follow the trajectory $\mathbf{y}$ with probability
	$$\Prob(\mathbf{Y}=\mathbf{y})-\sum_{\mathbf{x}:\ \mathbf{x}\sim\mathbf{y}} q_{\mathbf{x}\mathbf{y}},$$
	independently, for all $\mathbf{x}\in\mathcal{X}$ and $\mathbf{y}\in\mathcal{Y}$.
	
	From the flow constraints, we know that all these probabilities are in $[0,1]$ and the
	resulting coupling satisfies $\Prob(\mathbf{X}\sim\mathbf{Y})\geq 2-\mathcal{S}_T(x,y)$.
	The theorem then follows by combining this inequality with \eqref{eq:oneway}.
\end{proof}
\vspace{1em}

One can observe that, as the left-hand side of \eqref{eq:strongduality} is a probability and hence
at most $1$, we must always have optimal separation at least $1$. Indeed, we can obtain separation
equal to $1$, for instance by taking $A_t=S$ for all $0\leq t \leq T$. Recall that a separating
sequence is called non-trivial if the probabilities that $X_t\in A_t$ for all $0\leq t\leq T$ given
$X_0 = x$ and $X_t \not\in A_t$ for all $0\leq t \leq T$ given $X_0=y$ are both non-zero. Clearly,
any trivial separating sequence has separation at most $1$, so it follows from Theorem
\ref{thm:strongduality} that for any finite state Markov chain in discrete time, any $x, y$ and
$T$ as above, the following two statements are equivalent:
\begin{enumerate}[(a)]
	\item The meeting probability under optimal coupling of two copies, started in $x,y$
	respectively, is $1$.
	\item For all non-trivial separating sequences $\mathcal{A}=(A_t)_{t=0}^T$, we get
	$\mathcal{S}_T^{\mathcal{A}}(x,y) \leq 1$.
\end{enumerate}

\vspace*{.5em}
In order to get acquainted with the idea behind the concept of separation, let us take a look
at the simplest non-trivial example:

\begin{example}
	Let $0 < \alpha \leq \frac{1}{2}$. Consider the Markov chain $\mathbf{X}$ with state space
	$\{0, 1\}$ and transition probabilities $P(0,1) = P(1,0) = \alpha$ as well as $P(0,0)=P(1,1) = 1-\alpha$,
	and take $x=0,\ y=1$. As mentioned above, the case where $T=2$ is H\"aggstr\"om's \cite{disagreement}
	example of a segregating Markov chain.
	
	Since this chain only has two states, any non-trivial separating sequence must have $A_t= \{0\}$ or
	$A_t=\{1\}$ for each $0\leq t\leq T$. As $\alpha \leq \frac{1}{2}$, the non-trivial separating sequence
	given by $A_t=\{0\}$ for all $t$ is obviously best possible. It is immediate to check that its
	separation equals $2(1-\alpha)^T$. Hence, the optimal meeting probability of two copies, started in
	states $0$ and $1$ respectively, is $1$ if and only if 	$2(1-\alpha)^T \leq 1$.
	Using induction, one can easily check that for this chain we have
	$$ \lVert P^T(0,.)- P^T(1,.)\rVert_{\mathrm{TV}} = (1-2\alpha)^T.$$
	So by choosing $\alpha =  \alpha(T)$ such that $2(1-\alpha)^T=1$, we obtain a Markov chain that
	segregates the states $0$ and $1$ with total variation distance $(2^{1-1/T} - 1)^T$, which tends
	to $\frac{1}{4}$ as $T\rightarrow\infty$.
\end{example}

The next example is supposed to illustrate that reducible Markov chains obtained from irreducible
and aperiodic finite chains in finite time, in the way described before Theorem \ref{lowerboundthm},
usually do segregate any two states:

\begin{example}
	Let $\mathbf{X}$ be any irreducible aperiodic Markov chain with a finite state space $S$, and let $x$
	and $y$ be any two states. Pick $\epsilon>0$ and $n\in\N$ such that $P^n(x', y') \geq \epsilon$
	for all $x',y'\in S$. Then, for any non-trivial separating sequence $(A_t)_{t=0}^{nk}$, we have
	$$\mathcal{S}^{\mathcal{A}}_{nk}(x,y) \leq 2\,(1-\epsilon)^k$$
	for any $k\in\N$. Hence, by taking $T=nk$ for a sufficiently large $k$ it follows that the optimal
	meeting probability during $[0, T]$ is $1$. This shows that, unless
	$\lVert P^T(x,.)- P^T(y,.)\rVert_{\mathrm{TV}}=0$, the Markov chain segregates the two states $x$ and
	$y$, choosing $T$ sufficiently large.
\end{example}
\vspace{.5em}

We now turn to the proof of Theorem \ref{lowerboundthm}. Let $\mathbf{X}$ be the Markov chain with state
space $\{0, 1, \dots, L\}$ for some positive integer $L$ and transition probabilities given by
$P(0, 1) = P(L, L-1) = 1-P(0, 0) = 1-P(L, L) = \alpha$ as well as $P(i, i+1) = P(i, i-1) = \frac{1}{2}$
for all $0< i < L$, see Figure \ref{badc}. Such chains, with $S=\{0, 1, \dots, L\}$ and the additional
property that $|X_{t+1}-X_t|\leq 1$ a.s.\ for all $t$, are commonly called finite {\em birth-and-death chains},
cf.\ Section 2.5 in \cite{Mixing}. To begin with, our main interest lies in the optimal meeting probability
of this chain given the starting states $x=0$ and $y=L$. We will then show that we can obtain segregation
between $0$ and $L$ with total variation arbitrarily close to $\frac{1}{2}$ by choosing $L$, $T$ and
$\alpha$ appropriately.
% In our analysis below, we will consider a sequence of such Markov chains with $L$ fixed, $T$ tending
% to infinity and $\alpha = \Theta\left(\frac{L}{T}\right)$.

\begin{figure}[H]
	\centering\vspace*{-0.5em}
	\includegraphics[width=\textwidth]{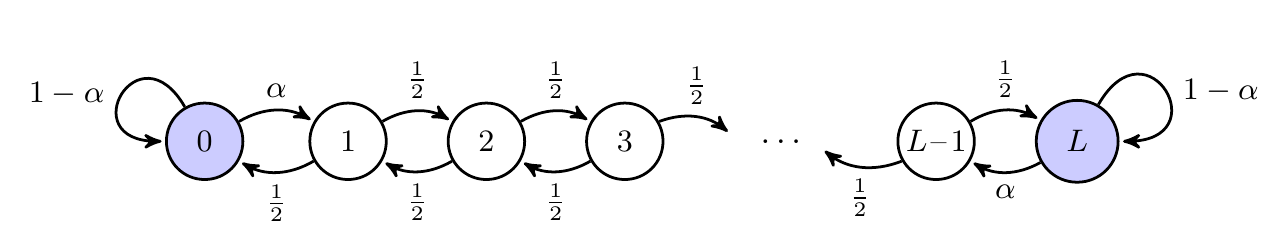}
	\caption{A birth-and-death chain on $\{0, 1, \dots, L\}$ as basic building block.}\label{badc}
\end{figure}

The qualitative behavior of this chain for small $\alpha$ is easy to describe: Most of the time, the
process is either at $0$ or $L$. Occasionally, that is at rate $\alpha$, the process takes one step
inwards and typically moves around for order $L$ steps before hitting one of the marginal states
$0$ or $L$, more precisely the expected time to reach $\{0,L\}$, when starting in state $1$ equals $L-1$.
With probability $1-\frac{1}{L}$ it will return to the same side it started at, with probability
$\frac{1}{L}$ it will cross over to the opposite side (check the analysis of the so-called
gambler's ruin problem in Section 2.1 in \cite{Mixing} for the explicit calculations).

%We begin our proof by listing some properties about the Markov chain $\mathbf{X}$. 
%\begin{fact}\label{fact:halfchain}
%	\begin{enumerate}
%		\item For any $0 < i < L$ and any $t\geq 0$ we have $P^t(0, i)=O(\alpha)$.
%		\item For any $1 \leq k \leq L$, conditioned on $\mathbf{X}_0=1$ or $L-1$ the probability that the chain hits $k$ before $0$ or $L-k$ before $L$ respectively is $\frac{1}{k}$.
%		\item For ``small $\alpha$'', the time intervals where $0 < \mathbf{X}_t < L$ are approximately Poisson distributed. In particular, for any $t\geq 0$, we have
%		\begin{equation}\label{eq:P00LL}
%		P^t(0, 0) = P^t(L, L) = \frac{1}{2} + \frac{1}{2} e^{-2\alpha t/L} + O\left( \alpha^2 t L + L\alpha\right),
%		\end{equation}
%		\begin{equation}\label{eq:XTV}
%		\|P^t(0, \cdot) - P^t(L, \cdot)\|_{TV} = e^{-2\alpha t/L}+O\left( \alpha^2 t L + L\alpha\right),
%		\end{equation}
%		and for any $0\leq k \leq L-1$
%		\begin{equation}\label{eq:Xbound}\begin{split}
%		&\mathbb{P}\left( \mathbf{X}_{t'} \leq k \text{ for all }0\leq t' \leq t\,\middle\vert\, \mathbf{X}_0=0\right)\\
%		&\qquad= \mathbb{P}\left( \mathbf{X}_{t'} \geq L- k \text{ for all }0\leq t' \leq t\,\middle\vert\, \mathbf{X}_0=L\right)\\
%		&\qquad = e^{-\alpha T/(k+1)} + O(\alpha^2 t L + L\alpha).
%		\end{split}\end{equation}
%	\end{enumerate}
%\end{fact}

In preparation to the in-depth analysis of the separation of states $0$ and $L$, let us
collect a few general estimations for this chain in the following lemma, which will come in useful later on.
For the sake of clarity we will use the standard big $O$ notation to represent error terms, i.e.\ for any
non-negative function $f$ in $k$ and $\alpha$, the expression $O(f(k,\alpha))$ denotes a quantity that is
bounded in absolute value by $c\cdot f(k,\alpha)$, where the constant $c>0$ does not depend on $k,\ \alpha$
or $t$.

\begin{lemma}\label{Fact1.3alt}
	\begin{enumerate}[(a)]
		\item For any $1 \leq  i \leq  L-1$ and $t\geq 0$ we have $P^t(0, i)=P^t(L, L-i)\leq 2\alpha$.
		\item For any $t\geq 0$ we have
		\begin{equation}\label{eq:P00LLalt}
		P^t(0, 0) = P^t(L, L) = \frac{1}{2} + \frac{1}{2} \left(1-\frac{2\alpha}{L}\right)^t + O(L\alpha),
		\end{equation}
		\begin{equation}\label{eq:TValt}
		\lVert P^t(0,.)- P^t(L,.)\rVert_{\mathrm{TV}} = \left(1-\frac{2\alpha}{L}\right)^t + O(L\alpha),
		\end{equation}
		\item For any $0\leq k \leq L-1$ and $t\geq 0$ we have
		\begin{equation}\label{eq:keepsmallalt}\begin{split}
		&\mathbb{P}\left( X_{t'} \leq k \text{ \upshape for all }0\leq t' \leq t\,|\, X_0=0\right)\\
		&\qquad= \mathbb{P}\left( X_{t'} \geq L- k \text{ \upshape for all }0\leq t' \leq t\,|\, X_0=L\right)\\
		&\qquad = \left( 1 - \frac{\alpha}{k+1}\right)^t + O(k\alpha).
		\end{split}\end{equation}
	\end{enumerate}
\end{lemma}
\begin{proof}\hfill
	\begin{enumerate}[(a)]
		\item  The first statement easily follows by induction on $t$ using the recursion
		$$P^{t+1}(0, i) = P^t(0, i-1)\, P(i-1, i) + P^t(0, i+1)\, P(i+1, i),$$
		which holds for any $t\geq 0$ and any $1\leq i \leq L-1$.
		
		\item To show the second claim, consider the sequence
		$a_t = \mathbb{E}\left[ X_t \,|\, X_0=0\right]$. Using part (a), we know that 
		\begin{equation}\label{eq:atPrel}
		a_t = L\cdot P^t(0, L)+O(L^2\alpha)=L-L\cdot P^t(0, 0) + O(L^2\alpha),
		\end{equation}
		where the error terms are bounded by $2\,L^2\alpha$ in absolute value.
		Furthermore, since $\mathbb{E}\left[ X_{t+1}-X_t\,|\, X_t\right]$ equals $\alpha$ if $X_t=0$, $-\alpha$ if $X_t=L$ and $0$ otherwise, we can infer from \eqref{eq:atPrel}
		\begin{align*}
		a_{t+1} &= a_t + \alpha\cdot P^t(0, 0) - \alpha\cdot P^t(0, L)\\
		&= a_t + \alpha\,\left(1 - \frac{a_t}{L} \right) - \alpha\, \frac{a_t}{L} + O(L\alpha^2),
		\end{align*}
		which implies 
		\begin{align*}
		a_{t+1}-\frac{L}{2} = \left(1-\frac{2\alpha}{L}\right)\cdot\left(a_t - \frac{L}{2}\right) + O(L\alpha^2),
		\end{align*}
		where the error term is bounded by $4\,L\alpha^2$ in absolute value irrespectively of $t$.
		Solving this recursion, using $a_{0} = 0$ and $\sum_{k=0}^{t-1}(1-\tfrac{2\alpha}{L})^k\leq\tfrac{L}{2\alpha}$,	
		yields $$a_t = \frac{L}{2} - \frac{L}{2}\,\left(1-\frac{2\alpha}{L}\right)^t + O(L^2\alpha).$$ The estimate \eqref{eq:P00LLalt} immediately follows from \eqref{eq:atPrel}, which together with part (a) implies
		\eqref{eq:TValt}.
		
		\item The case $k=0$ is obvious, so we may assume $k>0$. Let $\tau$ denote the first time $t\geq 0$ for
		which $X_t = k+1$. Consider the sequence $b_t = \mathbb{E}\left[X_{t\wedge \tau}\,|\, X_0=0\right]$, where
		$t\wedge \tau$ denotes the minimum of $t$ and $\tau$. Note that part (a) implies
		$\mathbb{P}\left( X_{t\wedge\tau} = i \,|\, X_0=0\right) \leq \mathbb{P}\left( X_t = i \,|\, X_0=0\right)
		\leq 2\alpha$ for any $i\in\{1,\dots,k\}$ and further 
		\begin{equation}\label{eq:btPrel}
		\begin{split}
		b_t &= (k+1)\cdot \mathbb{P}\left( X_{t\wedge\tau} = k+1 \,|\, X_0=0\right) + O(k^2\alpha)\\
		&= k+1 - (k+1)\cdot \mathbb{P}\left( X_{t\wedge\tau} = 0 \,|\, X_0=0\right) + O(k^2\alpha).
		\end{split}
		\end{equation}
		With the same reasoning as in part (b), we end up in a similar situation with $b_0=0$ and $b_t$ satisfying
		the recursive formula
		\begin{align*}
		b_{t+1} &= b_t + \alpha\cdot\mathbb{P}\left( X_{t\wedge\tau} = 0 \,|\, X_0=0\right)\\
		&=b_t + \alpha \left( 1 - \frac{b_t}{k+1}\right) + O(k \alpha^2).
		\end{align*}
		Solving it gives $b_t = k+1 - (k+1)\,( 1 - \frac{\alpha}{k+1})^t + O(k^2 \alpha)$ and plugging this into \eqref{eq:btPrel} completes the proof of part (c), using 
		$$\mathbb{P}\left( 1\leq X_{t\wedge\tau}\leq k \,|\, X_0=0\right)=O(k\alpha)$$ and noting that
		$X_{t\wedge\tau}\leq k$ if and only if $X_{t'}\leq k$ for all $0\leq t'\leq t$. %\vspace*{-1em}
	\end{enumerate}
\end{proof}

Let us now take a closer look on the optimal separation of states $0$ and $L$ in this chain. To make
our lives easier, we establish three auxiliary results showing that among the non-trivial separating
sequences, there are very simple ones which are essentially best possible as $T$ grows large.

\begin{proposition}\label{prop:boundariesnice}
	Let $L$ be fixed, and let $\alpha=\alpha(T) = \Theta\left(\frac{1}{T}\right)$. Then, for
	sufficiently large $T$, any non-trivial separating sequence $\mathcal{A}=(A_t)_{t=0}^T$
	such that $\mathcal{S}^{\mathcal{A}}_T(0,L)>1$ (if such exist) must satisfy $0\in A_t$ and
	$L\notin A_t$ for all $0\leq t \leq T$.
\end{proposition}
\begin{proof} By \eqref{eq:P00LLalt}, we know that $P^t(0, 0) = P^t(L, L) > 1/2$ for any
	$0\leq t \leq T$ given $T$ sufficiently large. Hence if $0 \not\in A_{t_1}$ and $L\in A_{t_2}$
	for some $0\leq t_1, t_2 \leq L$, then
	$$ \mathcal{S}^{\mathcal{A}}_T(0,L) \leq 1-P^{t_1}(0, 0) + 1 - P^{t_2}(L, L) < 1.$$
	So for $T$ sufficiently large, either $0\in A_t$ for all $0\leq t \leq T$ or $L\not\in A_t$
	for all $0\leq t \leq T$.
	
	By symmetry, we can assume without loss of generality, that $0 \in A_t$ for all $0\leq t \leq T$.
	Note that if there is some $t_1$ such that $L \in A_{t_1}$, then by part (a) of Lemma \ref{Fact1.3alt},
	we get
	$$\mathbb{P}\left( X_t\not\in A_t\text{ for all }0\leq t\leq T\,|\, X_0=L\right) 
	\leq \mathbb{P}\left( 0 < X_{t_1} < L \,|\, X_0=L\right) \leq 2L\alpha.$$
	So the proposition follows if we can show that there exists a constant $\epsilon>0$ such that
	for $T$ sufficiently large any non-trivial separating sequence satisfies
	$$ \mathbb{P}\left(X_t \in A_t\text{ for all }0\leq t \leq T\,|\, X_0=0\right) \leq 1-\epsilon.$$
	
	First note that since $\mathcal{A}$ is non-trivial, there exists a trajectory
	$\mathbf{y}\in \{L\}\times S^T$ such that $y_t\not\in A_t$ for all $0\leq t \leq T$ and
	$\Prob(\mathbf{X}=\mathbf{y}\,|\,X_0=L)>0$. Further, recall that the chain can only attain $\mathbf{y}$
	with positive probability if $|y_{t+1}-y_t|\leq 1$ for all $0\leq t\leq T-1$.

	Next, let $\mathbf{X}=(X_t)_{t=0}^{T+1}$ be a copy of the Markov chain started in $X_0=0$.
	We define the process $\mathbf{X}'=(X'_t)_{t=0}^T$ as $X'_t = X_{t+1}$ for all $0\leq t\leq T$ if
	$X_1=0$, and otherwise put $X'_0=0$ and let this process evolve independently of $\mathbf{X}$. Clearly,
	this implies that $(X_t)_{t=0}^{T}$ and $\mathbf{X}'$ have the same distribution.

	Now, if $X_{t} = L$ for some $0\leq t \leq T$, then the trajectory of $\mathbf{X}$ and $\mathbf{y}$
	necessarily either intersect or cross. Consequently, we either have $X_1\neq 0$ (which occurs with probability
	$\alpha$) or at least one of $\mathbf{X}$ and $\mathbf{X}'$ meets $\mathbf{y}$, and is thus outside $A_t$ for
	some $t$. By the union bound, we find
	\begin{equation*}
	\begin{split}
	&\mathbb{P}\left( X_t = L\text{ for some }0\leq t \leq T\,|\, X_0=0\right)\\
	&\qquad\leq \alpha + 2\,\mathbb{P}\left( X_t \not\in A_t\text{ for some }0\leq t\leq T\,|\, X_0=0 \right).
	\end{split}
	\end{equation*}
	From \eqref{eq:keepsmallalt} we know that $\Prob(X_t\leq L-1\text{ for all }0\leq t \leq T\,|\, X_0=0)$
	is bounded away from $1$ as $T$ tends to infinity with our choice of $\alpha=\Theta(\tfrac1T)$. So choosing
	$\epsilon>0$ small, $T$ large enough such that
	$$\mathbb{P}\left( X_t = L\text{ for some }0\leq t \leq T\,|\, X_0=0\right)\geq \alpha +2\epsilon$$
	will do the job.
\end{proof}

\begin{proposition}\label{prop:cyclicperm}
	Let $\mathcal{A}=(A_t)_{t=0}^T$ be any separating sequence such that $0 \in A_t$ and $L\not\in A_t$
	for all $0\leq t \leq T$. For any $0\leq a \leq T$, we define the separating sequence $\mathcal{A}^a
	=(A_t^a)_{t=0}^T$ by $A^a_t:= A_{t+a\;(\text{\upshape mod }T+1)}$. Then
	$$ \mathcal{S}^{\mathcal{A}}_T(0,L) \leq \mathcal{S}^{\mathcal{A}^a}_T(0,L)+12\,L\alpha.$$
\end{proposition}
\begin{proof} The case where $a=0$ is obvious, so we may assume $a>0$. By part (a) of Lemma \ref{Fact1.3alt},
	for any fixed $0\leq t\leq T$, the probability that $X_t \in A_t\setminus \{0\}$ given $X_0=0$ is at
	most $2\,L\alpha$. From this we can infer
	\begin{align*}
	&\mathbb{P}\left( X_t \in A_t\text{ for all }0\leq t\leq T \,|\, X_0=0\right) \\
	&\quad\leq\mathbb{P}\left( X_t \in A_t\text{ for all }0\leq t\leq T\text{ and }
	X_{a-1}=X_a=X_T=0 \,|\, X_0=0\right)+6\,L\alpha\\
	&\quad= \mathbb{P}\left( X_t \in A_t\text{ for all }0\leq t\leq a-1 \text{ and }
	X_{a-1}=0 \,|\, X_0=0\right)\cdot(1-\alpha)\\
	&\quad\qquad\cdot \mathbb{P}\left( X_t \in A_t\text{ for all }a\leq t\leq T \text{ and }
	X_{T}=0 \,|\, X_a=0\right) +6\,L\alpha\\
	&\quad= \mathbb{P}\left( X_t \in A_{t-T+a-1}\text{ for all }T-a+1\leq t\leq T \text{ and }
	X_T=0 \,|\, X_{T-a+1}=0\right)\\
	&\quad\qquad\cdot(1-\alpha)\cdot \mathbb{P}\left( X_t \in A_{t+a}\text{ for all }0\leq t\leq T-a
	\text{ and }X_{T-a}=0 \,|\, X_0=0\right)\\
	&\quad\qquad+ 6\,L\alpha\\
	&\quad=\mathbb{P}\left( X_t \in A^a_t\text{ for all }0\leq t\leq T\text{ and }
	X_{T-a}=X_{T-a+1}=X_T = 0\,|\, X_0=0\right)\\
	&\quad\qquad+ 6\,L\alpha\\
	&\quad\leq \mathbb{P}\left( X_t \in A^a_t\text{ for all }0\leq t\leq T \,|\, X_0=0\right) + 6\,L\alpha,
	\end{align*}
	where the first and last equality follow from the Markov property and the central one from
	time homogeneity of the chain.
	
	By symmetry, the same argument works for the chain started at $L$ and the sequence of complementary
	sets $(S\setminus A_t)_{t=0}^T$.
\end{proof}

\begin{proposition}\label{prop:constantalmostbest}
	For any separating sequence $\mathcal{A}=(A_t)_{t=0}^T$ such that $0\in A_t$ and $L\not\in A_t$
	for all $0\leq t \leq T$, there exists a $k\in\{0,\dots,L\}$ such that
	\begin{equation}
	\mathcal{S}_T^{\mathcal{A}}(0,L) \leq \mathcal{S}_T^{\underline{\mathbf{k}}}(0,L) + 12\,L \alpha,
	\end{equation}
	where $\underline{\mathbf{k}}$ denotes the constant separating sequence whose elements are all equal to the
	set $\{0, 1, \dots, k\}$.
\end{proposition}
\begin{proof}
	By Proposition \ref{prop:cyclicperm}, we have
	\begin{align*}
	\mathcal{S}_T^{\mathcal{A}}(0,L) \leq \frac{1}{T+1}\,\sum_{a=0}^T \mathcal{S}^{\mathcal{A}^a}_T(0,L)
	+12\,L\alpha.
	\end{align*}
	For any $0\leq k \leq L$, let us define
	$$f(k) := \min_{0\leq k' \leq k}\frac{\abs{ \{t : k' \in A_t\}}}{T+1}.$$
	Note that the function $f$ is decreasing, $f(0)=1$ and $f(L)=0$. To simplify the notation, we additionally
	set $f(L+1) := 0$ and write $M:=\max_{0\leq t\leq T} X_t$. Considering only the first summands coming from
	each of the separation terms $\mathcal{S}^{\mathcal{A}^a}_T(0,L),\ 0\leq a\leq T$, we find
	\begin{align*}
	&\frac{1}{T+1}\, \sum_{a=0}^T\mathbb{P}\left( X_t \in A^a_t \text{ for all }0\leq t\leq T \,
	\middle\vert\, X_0=0\right)\\
	&\qquad = \frac{1}{T+1}\, \sum_{a=0}^T \sum_{k=0}^L \mathbb{P}\left( M = k \,
	\middle\vert\, X_0=0\right)\\
	&\qquad\qquad\cdot  \mathbb{P}\left( X_t \in A^a_t \text{ for all }0\leq t\leq T
	\,\middle\vert\,M = k \text{ and } X_0=0\right)
	\end{align*}
	
	Fix $k\in\{0,\dots,L\}$, pick $k'$ to minimize $\abs{\{t: k'\in A_t\}}$ over $\{0,\dots,k\}$ and note
	that this implies
	$$f(k)=\frac{\abs{\{t: k'\in A_t\}}}{T+1}=	\frac{1}{T+1}\,\sum_{a=0}^T \mathbbm{1}_{\{k'\in A_t^a\}}.$$
	Let $\tau\geq 0$ be the first time when the Markov chain visits state $k'$. Given $M=k$, we know
	$\tau\leq T$, hence
	\begin{align*}
	&\frac{1}{T+1}\, \sum_{a=0}^T \mathbb{P}\left( X_t \in A^a_t \text{ for all }0\leq t\leq T \,
	\middle\vert\, M = k \text{ and } X_0=0\right)\\
	&\qquad \leq \frac{1}{T+1}\, \sum_{a=0}^T \mathbb{P}\left( k' \in A^a_\tau \,
	\middle\vert\,M = k \text{ and } X_0=0\right)\\
	&\qquad = \frac{1}{T+1}\, \sum_{a=0}^T\sum_{t=0}^T \mathbb{P}\left( \tau=t \,
	\middle\vert\,M = k \text{ and } X_0=0\right)\cdot \mathbbm{1}_{\{k'\in A_t^a\}}\\
	&\qquad=\sum_{t=0}^T \mathbb{P}\left( \tau=t \,\middle\vert\, M = k \text{ and } X_0=0\right)
	\cdot f(k)= f(k).
	\end{align*}
	We conclude that
	\begin{align*}
	&\frac{1}{T+1} \sum_{a=0}^T\mathbb{P}\left( X_t \in A^a_t \text{ for all }0\leq t\leq T \,
	\middle\vert\, X_0=0\right)\\
	&\qquad \leq \sum_{k=0}^L \mathbb{P}\left( \max_{0\leq t\leq T} X_t = k \,
	\middle\vert\, X_0=0\right) \cdot f(k)\\
	&\qquad = \sum_{k=0}^L \mathbb{P}\left( \max_{0\leq t\leq T} X_t \leq k \,
	\middle\vert\, X_0=0\right) \cdot\big( f(k) - f(k+1)\big).
	\end{align*}
	
	Arguing analogously in the case of $X_0=L$, with the modifications that we consider
	$\min_{0\leq t\leq T} X_t$ instead of $\max_{0\leq t\leq T} X_t$ and $\tau$ now denotes the first
	time when the chain is in state $k$, yields
	\begin{align*}
	&\frac{1}{T+1} \sum_{a=0}^T\mathbb{P}\left( X_t \not\in A^a_t \text{ for all }0\leq t\leq T \,
	\middle\vert\, X_0=L\right)\\
	&\qquad \leq \sum_{k=0}^L \mathbb{P}\left( \min_{0\leq t\leq T} X_t = k \,\middle\vert\, X_0=L\right)
	\cdot \big( 1- f(k)\big)\\
	&\qquad = 1-\sum_{k=0}^L \mathbb{P}\left( \min_{0\leq t\leq T} X_t \leq k \,\middle\vert\, X_0=L\right) 
	\cdot \big( f(k) - f(k+1)\big).\\
	&\qquad = \sum_{k=0}^L \mathbb{P}\left( \min_{0\leq t\leq T} X_t > k \,\middle\vert\, X_0=L\right) 
	\cdot \big( f(k) - f(k+1)\big),
	\end{align*}
	where we used $1-f(k)\geq 1-\tfrac{1}{T+1}\cdot\abs{\{t: k\in A_t\}}$ to derive the inequality and the last
	equality follows from $\sum_{k=0}^L f(k) - f(k+1) = 1$.	By combining these two estimates, it follows that
	\begin{align*}
	\frac{1}{T+1}\sum_{a=0}^T \mathcal{S}_T^{\mathcal{A}^a}(0,L) \leq 
	\sum_{k=0}^L \big( f(k)-f(k+1) \big)\cdot\mathcal{S}^{\underline{\mathbf{k}}}_T(0,L).
	\end{align*}
	From the fact that the coefficients $f(k) - f(k+1),\ 0\leq k\leq L,$ sum up to $1$, plus $f(L) = f(L+1)$,
	we can conclude that there exists some $k\in\{0,\dots,L-1\}$ such that
	\begin{equation*}
	\frac{1}{T+1}\sum_{a=0}^T \mathcal{S}_T^{\mathcal{A}^a}(0,L) \leq \mathcal{S}^{\underline{\mathbf{k}}}_T(0,L),
	\end{equation*}
	which completes the proof.
\end{proof}
\vspace*{1em}

Combining Propositions \ref{prop:boundariesnice} and \ref{prop:constantalmostbest}, it follows that for any fixed $L\geq 1$, any $\alpha = \alpha(T) = \Theta\left(\frac{1}{T}\right)$ and $T$ sufficiently large, the optimal separation $\mathcal{S}_T(0,L)$ is the maximum of $1$ and
\begin{equation}\label{eq:bestnontrivial}
\mathcal{S}^{\underline{\mathbf{k}}}_T(0,L) + O(L\alpha) \leq
e^{-\alpha T / (k+1)} + e^{-\alpha T/(L-k)} + O(L\alpha),
\end{equation}
for $0 \leq k < L$, where the inequality follows from \eqref{eq:keepsmallalt}.

To finish the proof of Theorem \ref{lowerboundthm}, we need one more elementary estimation:
\begin{lemma}\label{cd}
	Given $A > 0$, define the function
	$$f_A(x) := e^{-A/x} + e^{-A/(1-x)},\quad x\in (0, 1).$$
	Then it holds $$\sup_{0<x<1}f_A(x)= \max\left( e^{-A}, 2\, e^{-2A}\right).$$
\end{lemma}
\begin{proof} 
	First note that the function $f_A$ lies in $C^\infty((0,1))$ and is symmetric around $x=\frac12$.
	Calculating its first two derivatives gives
	\begin{align*}
	f_A'(x)&=\tfrac{A}{x^2}\cdot \text{e}^{-\tfrac{A}{x}}-\tfrac{A}{(1-x)^2}\cdot \text{e}^{-\tfrac{A}{1-x}}
	\quad\text{and}\\
	f_A''(x)&=\big(\tfrac{A^2}{x^4}-\tfrac{2A}{x^3}\big)\cdot \text{e}^{-\tfrac{A}{x}}
	+\big(\tfrac{A^2}{(1-x)^4}-\tfrac{2A}{(1-x)^3}\big)\cdot \text{e}^{-\tfrac{A}{(1-x)}}.
	\end{align*}
	%	It has critical points at $x=\frac{1}{2}$ and solutions to
	%	$$ e^{A} = \left(\frac{x}{1-x}\right)^{\tfrac{2x \,(1-x)}{2x-1}}.$$
	%	One can check that the expression to the right is again symmetric around $x=\frac12$, increasing on
	%	the interval $(0, \frac{1}{2})$ and decreasing on $(\frac{1}{2}, 1)$ with a limit of
	%	$\text{e}$ at $x=\frac{1}{2}$.
	%	
	%	Hence, if $A \geq 1$, then the only critical point is $\frac{1}{2}$ and the function either attains
	%	its maximum at $x=\frac{1}{2}$, where it equals $f_A(\tfrac12)=2\,e^{-2A}$, or tends to its
	%	supremum at the boundary, where it converges to $e^{-A}$. In the case $A< 1$, however,
	%	$$f_A''\left(\tfrac{1}{2}\right) = 32\,(A^2-A)e^{-2A} < 0,$$ so $x=\frac{1}{2}$ is a local maximum.
	%	As $f_A$ can have at most two additional critical points in the interval $(0,1)$, one on either side of $\frac{1}{2}$,
	%	these have to be local minima. 
	If $x\in(0,1)$ is a stationary point of $f_A$, we necessarily have 
	$$\tfrac{A}{x^2}\cdot \text{e}^{-\tfrac{A}{x}}=\tfrac{A}{(1-x)^2}\cdot \text{e}^{-\tfrac{A}{1-x}}$$ and
	as a consequence the sign of $f_A''(x)$ is given by the sign of 
	$$g_A(x):=2\,(A+1)\,x^2-2\,(A+1)\,x+A.$$
	Due to $A>0$, the function $g_A$ is strictly convex. Assuming the existence of two local maxima
	of $f_A$ on $(0,1)$ -- at points $x_1<x_2$ say -- forces the existence of a local minimum at $x_3\in(x_1,x_2)$.
	Hence $\max\{f_A''(x_1),f_A''(x_2)\}\leq0\leq f_A''(x_3)$, which contradicts the strict convexity
	of $g_A$.
	
	Consequently, $f_A$ can have at most one local maximum in $(0,1)$, which then lies at $x=\tfrac12$
	for symmetry reasons. In conclusion, $f_A$ either attains its maximum on $(0,1)$ at
	$x=\frac{1}{2}$ or converges to its supremum on the boundary.
\end{proof}
\vspace*{1em}

\begin{proof}[of Theorem \ref{lowerboundthm}]
	Applying Lemma \ref{cd} to \eqref{eq:bestnontrivial} with $A=\alpha T/(L+1)$, we see that we can pick
	$\delta>0$ arbitrarily small, choose
	$$\alpha=\alpha(T)= \frac12\,\big(\ln(2)+\delta\big)\cdot\frac{L+1}{T}$$ 
	and find that the optimal separation of the states $0$ and $L$ in the chain $\mathbf{X}$ on $[0,T]$ is $1$
	if $T$ is chosen sufficiently large (after having fixed $L$). Hence, by Theorem \ref{thm:strongduality} and 
	\eqref{eq:TValt} we know that the chain $(X_t)_{t=0}^T$ segregates $0$ and $L$. From \eqref{eq:TValt} we can
	further read off that for $L$ fixed
	\begin{equation}\label{TVlimit}
	\lVert P^T(0,.)- P^T(L,.)\rVert_{\mathrm{TV}}=\text{e}^{-\big(\ln(2)+\delta\big)\,\tfrac{L+1}{L}}-o_T(1).
	\end{equation}
	Given $\epsilon>0$, we can choose $\delta>0$ small, $L$ large enough and then pick $T$ sufficiently large
	to make the right-hand side of \eqref{TVlimit} larger than $\frac12-\epsilon$. This completes the proof.
\end{proof}

\begin{remark}
	One downside of the implicit construction proving Theorem \ref{lowerboundthm} is the fact that it does
	not give much information about the coupling involved. As the coupling will have to take into account
	the whole trajectories of the two individual copies, it is highly unlikely that the coupled process
	will have the Markov property. In this respect, it is still an open problem if the value of $\frac{1}{\mathrm{e}}$,
	established in Section \ref{exampl}, can be pushed further (as supremum of achievable total variation
	distances that can be retained in segregating Markov chains), if we restrict ourselves to {\em Markovian}
	couplings.
	
	We can however rule out that there exists a single chain {\em in discrete time} with two segregated states $x$ and
	$y$ such that $$\lim_{n\to\infty}\lVert P^n(x,.)-P^n(y,.)\rVert_{\mathrm{TV}}= \tfrac12,$$
	i.e.\ for which the value $\frac12$ actually is attained (cf.\ the following proposition, which
	slightly improves the result from Theorem \ref{upperboundthm}).	
\end{remark}

\begin{proposition}\label{strict}
	Consider a Markov chain in discrete time with countable state space $S$ and two states $x,y\in S$. If two copies,
	$\mathbf{X}=(X_t)_{t\in\N_0}$ and $\mathbf{Y}=(Y_t)_{t\in\N_0}$, started in $x$ and $y$ respectively, can be coupled to meet almost surely in finite time, it follows that
	$$\lim_{n\to\infty}\lVert P^n(x,.)-P^n(y,.)\rVert_{\mathrm{TV}}< \tfrac12.$$
\end{proposition}

\begin{proof}
	As a matter of fact, we can alter the proof of Proposition \ref{coupl} to derive the above statement:
	For the reasoning there to work, we need a function $f: \N_0\times S\to [0,1]$, replacing
	the martingales $M_t$ and $N_t$, with the following two properties:
	\begin{enumerate}[(i)]
		\item $\big(f(t,X_t)\big)_{t\in\N_0}$ is a martingale with respect to the natural filtration of
		$\mathbf{X}$, and likewise for $\mathbf{Y}$.
		\item $f(0,x)-f(0,y)=\lim_{n\to\infty}\limits\lVert P^n(x,.)-P^n(y,.)\rVert_{\mathrm{TV}}$.
	\end{enumerate}
	In order to compile such a function, let us define the sets $A_n\subseteq S, n\in\N_0,$ via
	$$A_n:=\{s\in S;\ P^n(x,s)>P^n(y,s)\},$$
	which implies $\lVert P^n(x,.)-P^n(y,.)\rVert_{\mathrm{TV}}=P^n(x,A_n)-P^n(y,A_n)$, and further
	$f_n(t,s):=P^{n-t}(s,A_n)$ for all $t\leq n$. Finally, choose $f$ to be the limit of a pointwise converging
	subsequence of the uniformly bounded sequence of functions $(f_{n})_{n\in\N_0}$.
	
	Then (ii) is immediate and since for all $n\in\N_0$, $f_n(t,X_t)$ is a martingale for $0\leq t\leq n$, bounded by
	0 and 1 from below and above respectively, the conditional dominated convergence theorem ensures that $f(t,X_t)$ inherits
	these properties.
	
	Analogously to the the proof of Proposition \ref{coupl}, let us define
	$$B_x:=\{f(t,X_t)\geq \tfrac12\text{ for all }t\in\N_0\}\quad\text{and}\quad
	B_y:=\{f(t,Y_t)< \tfrac12\text{ for all }t\in\N_0\}$$
	as well as
	$$\tau_x := \inf\{t\in\N_0;\ f(t,X_t) < \tfrac{1}{2}\}\quad\text{and}\quad
	\tau_y := \inf\{t\in\N_0;\ f(t,Y_t) \geq \tfrac{1}{2}\}.$$
	Note that the almost sure limit of $f(t,X_t)$ as $t\to\infty$ exists, according to Doob's martingale
	convergence theorem, which implies that $f(\tau_x,X_{\tau_x})$ is well defined even on $B_x=\{\tau_x=\infty\}$.
	Further note that $B_x^{\,\text{c}}\cup B_y^{\,\text{c}}$ is an almost sure event, due to the fact that
	$\mathbf{X}$ and $\mathbf{Y}$ meet in finite time with probability 1.
	
	\vspace{1em}\noindent
	If $\mathbb{P}(B_x^{\,\text{c}})>0$, we get the strict inequality from \eqref{M_t} as
	$f(\tau_x,X_{\tau_x})<\tfrac12$ on $B_x^{\,\text{c}}$:
	\begin{align*}\lim_{n\to\infty}\lVert P^n(x,.)-P^n(y,.)\rVert_{\mathrm{TV}}&=f(0,x)-f(0,y)\\
	&=\mathbb{E}\,f(\tau_x,X_{\tau_x})-\mathbb{E}\,f(\tau_y,Y_{\tau_y})\\
	&<1-\tfrac12\,\left( \mathbb{P}(B_x^{\,\text{c}}) + \mathbb{P}(B_y^{\,\text{c}})\right)\\
	&=\tfrac12.
	\end{align*}
	
	If both $B_x^{\,\text{c}}$ and $\{f(t,Y_t)> \tfrac12\text{ for some }t\in\N_0\}$ have probability
	0, the fact that $\mathbf{X}$ and $\mathbf{Y}$ meet a.s.\ forces the events
	$\{f(t,X_t)=\tfrac12\text{ for some }t\in\N_0\}$ and $\{f(t,Y_t)=\tfrac12\text{ for some }t\in\N_0\big\}$
	to have probability 1. Changing the stopping times to
	$$\tau_x := \inf\{t\in\N_0;\ f(t,X_t) = \tfrac{1}{2}\}\quad\text{and}\quad
	\tau_y := \inf\{t\in\N_0;\ f(t,Y_t) = \tfrac{1}{2}\}$$ gives
	$$\lim_{n\to\infty}\lVert P^n(x,.)-P^n(y,.)\rVert_{\mathrm{TV}}
	=\mathbb{E}\,f(\tau_x,X_{\tau_x})-\mathbb{E}\,f(\tau_y,Y_{\tau_y})=0.$$
	Hence by symmetry it is safe to assume $\mathbb{P}(B_x^{\,\text{c}})>0$, which verifies the claim.
\end{proof}

\section*{Acknowledgements}
First of all, we want to thank Jeff Steif for bringing this special kind of coupling to our attention.
Then, we want to give thanks to him, Olle H\"aggstr\"om, Peter Hegarty and the anonymous referees for their valuable comments on earlier drafts.

% Non-BibTeX users please use

%\section*{About the author:}
%   We would like a short biographical sketch,
%   beyond just your affiliation to be placed
%   after the bibliography.
%   And below that, your full address.

\vspace{0.5cm}
\makebox[\textwidth][c]{
\begin{minipage}[t]{1.2\textwidth}
\begingroup
	\begin{minipage}[t]{0.5\textwidth}
	{\sc \small Timo Hirscher\\
		Department of Mathematical Sciences,\\
		Chalmers University of Technology\\
		and University of Gothenburg,\\
		412 96 Gothenburg, Sweden.}\\
	hirscher@chalmers.se
	\end{minipage}
	\hfill
	\begin{minipage}[t]{0.5\textwidth}
	{\sc \small Anders Martinsson\\
   Department of Mathematical Sciences,\\
   Chalmers University of Technology\\
   and University of Gothenburg,\\
   412 96 Gothenburg, Sweden.}\\
   andemar@chalmers.se\\
	\end{minipage}
	\endgroup
\end{minipage}}

\end{document}